\documentclass[12pt]{article}
\usepackage{amsthm}
\usepackage{amsmath}
\usepackage{amsfonts}
\usepackage{graphicx}
\usepackage{latexsym}
\usepackage{amssymb}

\newcommand{\deff}{\mbox{$\stackrel{\rm def}{=}$}}
\newcommand{\Span}[1]{{\left\langle {#1} \right\rangle}}

\newcommand{\field}[1]{\mathbb{#1}}

\newcommand{\F}{\field{F}}

\newcommand{\cA}{{\cal A}}
\newcommand{\cB}{{\cal B}}

\newcommand{\cD}{{\cal D}}

\newcommand{\cP}{{\cal P}}

\newcommand{\cH}{{\cal H}}
\newcommand{\cG}{{\cal G}}

\newtheorem{theorem}{Theorem}
\newtheorem{lemma}{Lemma}

\newtheorem{cor}{Corollary}
\newtheorem{conj}{Conjecture}

\makeatletter
 \DeclareRobustCommand{\nsbinom}{\genfrac[]\z@{}}
 \makeatother
 
 \newcommand{\sbinomq}[2]{\nsbinom{{#1}}{{#2}}_{q}}

\begin{document}

\bibliographystyle{plain}

\title{
\begin{center}
The $q$-Analog of the Middle Levels Problem
\end{center}
}
\author{
{\sc Tuvi Etzion}\thanks{Department of Computer Science, Technion,
Haifa 32000, Israel, e-mail: {\tt etzion@cs.technion.ac.il}.} }

\maketitle

\begin{abstract}
The well-known middle levels problem is to find a Hammiltonian cycle
in the graph induced from the binary Hamming graph $\cH_2(2k+1)$
by the words of weight $k$ or $k+1$.
In this paper we define the $q$-analog of the middle levels problem.
Let $n=2k+1$ and let~$q$~be a power of a prime number. Consider the set of
$(k+1)$-dimensional subspaces and the set of $k$-dimensional subspaces
of $\F_q^n$. Can these subspaces be ordered in a way that for any
two adjacent subspaces $X$ and $Y$, either $X \subset Y$ or $Y \subset X$?
A construction method which yields many Hamiltonian cycles for any given $q$ and $k=2$ is presented.
\end{abstract}

\vspace{0.5cm}


\footnotetext[1] { This research was supported in part by the Israeli
Science Foundation (ISF), Jerusalem, Israel, under
Grant 10/12.}

\newpage
\section{Introduction}
\label{sec:introduction}

Let $\cH_2(n)$ denote the $n$-dimensional binary Hamming graph, known also
as the $n$-dimensional hypercube. The graph has a set of $2^n$ vertices
which are represented by the set of all the binary $n$-tuples. Two vertices are adjacent if
the two $n$-tuples $(a_1,a_2,\ldots,a_n)$ and $(b_1,b_2,\ldots,b_n)$
which represent them differ in exactly one position.
The \emph{weight} of a binary $n$-tuple
is defined as the number of position in which it has \emph{ones}.
The \emph{middle levels graph}
$M_{2k+1}$ is a subgraph of $\cH_2(2k+1)$ induced by the vertices represented
by $n$-tuples of weight $k$ or weight $k+1$ in $\cH_2(2k+1)$. The
\emph{middle levels conjecture} asserts that the graph $M_{2k+1}$ has
a Hamiltonian cycle for every positive integer $k$.
This conjecture was formulated first by~\cite{Hav83}.
Using computer search it was verified independently by various sources that
$M_{2k+1}$ has a Hamiltonian cycle for $1 \leq k \leq 15$. This was also verified
for $k=16$ and $k=17$ in~\cite{SSS09} and for $k=18$ in~\cite{ShAm11}.
In~\cite{SaWi95} it was shown how to produce long cycles in $M_{2k+1}$ based
on a Hamiltonian cycle in a graph $M_{2k'+1}$ for which $k' < k$. Finally, in~\cite{Joh04} it
is shown that asymptotically a Hamiltonian cycle exists in $M_{2k+1}$, i.e., that there
exists a cycle of length $(1- o(1)) 2 \binom{2k+1}{k}$ in $M_{2k+1}$.

In this paper we are interested in the $q$-analog of the
middle levels problem. Let $\F_q^n$ be the vector space
of dimension $n$ over the finite field with $q$ elements $\F_q$.
Let $\cP_q(n)$ be the graph whose set of vertices represents the set of all
subspaces of $\F_q^n$. $\cP_q(n)$ is called the \emph{projective space graph}.
Two subspaces $X$ and $Y$ are connected by
an edge if $\dim X + \dim Y - 2 \dim (X \cap Y) =1$, i.e., the dimensions
of $X$ and $Y$ differ by one and either $X \subset Y$ or $Y \subset X$.
The \emph{middle levels} of $\cP_q(2k+1)$ is the graph which is induced
by the vertices (subspaces) of dimension $k$ and the
vertices (subspaces) of dimension $k+1$ from $\cP_q(2k+1)$.
The \emph{middle levels problem} for $\cP_q(2k+1)$
is to find a Hamiltonian cycle
in the middle levels of $\cP_q(2k+1)$.

One of the most important methods to construct long cycles in $M_{2k+1}$ was
introduced in~\cite{Dej85}.
A \emph{necklace} is a set which consists of an $n$-tuple
and all its distinct cyclic shifts.
Let $n=2k+1$ and assume
that there exists a sequence of $n$-tuples $P=x_1 , y_1 , x_2 , y_2 , \ldots , x_t , y_t$
satisfying the following requirements.
\begin{itemize}
\item $x_i$, $1 \leq i \leq t$, is an $n$-tuple with weight $k$.
\item $y_i$, $1 \leq i \leq t$, is an $n$-tuple with weight $k+1$.
\item All the $2t$ $n$-tuples in the sequence are contained in $2t$
different necklaces.
\item $x_i$ and $y_i$, $1 \leq i \leq t$, are connected by an edge in $M_{2k+1}$.
\item $y_i$ and $x_{i+1}$, $1 \leq i \leq t-1$, are connected by an edge in $M_{2k+1}$.
\item There exists a cyclic shift of $x_1$ by $\ell$ positions which yields an
$n$-tuple $x_1'$, such that $y_t$ and $x_1'$ are connected by an edge
in $M_{2k+1}$ and $\gcd ( \ell ,n)=1$.
\end{itemize}
We form the sequence $\Pi \deff P^0 =P , P^\ell , P^{2\ell} , P^{3 \ell} , \ldots$,
where superscripts are taken modulo $n$.
$P^i$ is the sequence formed from~$P$ by taking the cyclic shifts
by $i$ positions of all the $n$-tuples in~$P$, where the order
of the $n$-tuples in $P$ and $P^i$ is the same. Hence, the sequence $\Pi$ ends
with the subsequence $P^{n-\ell} =P^{(n-1)\ell}$ since $\gcd ( \ell ,n)=1$.

We suggest a method, akin to the one based
on necklaces to solve the middle levels problem
of $\cP_q(2k+1)$. The method is applied successfully
and very simply with $k=1$.
For $k=2$, we note that even though the value of $k$ is small,
the graph can be very large.
The number of two-dimensional subspaces of $\F_q^5$ is
$\frac{(q^5-1)(q^4-1)}{(q^2-1)(q-1)} = (q^4+q^3+q^2+q+1)(q^2+1)$ which
is getting larger as $q$ increases.
The method is also applied successfully in this case
as will be proved in this paper. It is worth to mention that
a similar $q$-analog problem to find universal cycles
of two-dimensional subspaces of $\cP_q(n)$ is solved
in~\cite{JBM09}. For this problem, the one-dimensional
subspaces of $\cP_q(n)$ are ordered cyclically in such a way that each
two-dimensional subspace of $\cP_q(n)$ is spanned by
exactly one pair of adjacent one-dimensional subspaces
and each one-dimensional subspace of $\cP_q(n)$ appears
at least once in the ordering.

The rest of this paper is organized as follows.
In Section~\ref{sec:cyclic_shifts}, we discuss the
representation of subspaces. We define the cyclic shifts
of subspaces and describe our method to form long cycles
in the middle levels of $\cP_q(2k+1)$. In
Section~\ref{sec:2_3subspaces}, we prove some properties
of two-dimensional subspaces and three-dimensional subspaces
of $\F_q^5$ which will be very useful when our method is applied
on $\cP_q(5)$. In Section~\ref{sec:Hamilton}, we will show how to construct
Hamiltonian cycles in the middle levels of $\cP_q(5)$.
Conclusion is given in Section~\ref{sec:conclusion}.

\section{Hamiltonian cycles based on cyclic shifts}
\label{sec:cyclic_shifts}

Given a nonnegative integer $r \leq n$, the set of all subspaces
of $\F_q^n$ that have dimension $r$ is known as a \emph{Grassmannian},
and usually denoted by $\cG_q(n,r)$. The number of subspaces in
$\cG_q(n,r)$ is
$$
| \cG_q(n,r) | = \sbinomq{n}{r} \deff \frac{(q^n-1)(q^{n-1}-1) \ldots (q^{n-r+1}-1)}{(q^r-1)(q^{r-1}-1) \ldots (q-1)}~,
$$
where $\sbinomq{n}{r}$ is the \emph{$q$-ary Gaussian coefficient}.
We note that the number of one-dimensional subspaces of $\F_q^n$
is $\frac{q^n-1}{q-1}$, where two vectors of $\F_q^n$ which are
multiples of each other by an element of $\F_q$ belong to the
same one-dimensional subspace.

Let $\F_{q^n}$ be a finite field with $q^n$ elements, where $q$ is
a power of a prime number,
and let $\alpha$ be a primitive element in $\F_{q^n}$.
It is well-known that there is an isomorphism between $\F_{q^n}$ and
$\F_q^n$, where the \emph{zero} elements are mapped into each other,
and $\alpha^i \in \F_{q^n}$, $0 \leq i \leq q^n-2$, is mapped into its
$q$-ary $n$-tuple representation in $\F_q^n$, and vice versa.
Assume that $\alpha_i, \alpha^j \in \F_{q^n}$
and their $n$-tuple representations are $X^{i} = (x^{i}_1,x^{i}_2 ,\ldots , x^{i}_n)$
and $X^{j} = (x^{j}_1,x^{j}_2 ,\ldots , x^{j}_n)$, respectively, where
$x^{i}_\ell , x^{j}_\ell \in \F_q$, $1 \leq \ell \leq n$.
The isomorphism implies that
the $n$-tuple $(x^{i}_1 + x^{j}_1,x^{i}_2+x^{j}_2 ,\ldots , x^{i}_n+x^{j}_n)$, where
the addition is done in $\F_q$, is the $n$-tuple representation of $\alpha^i + \alpha^j$.
Using this mapping, a subspace
of~$\F_q^n$ is represented by the corresponding elements of~$\F_{q^n}$.
Throughout this paper we will not distinguish between the two representations
and the vector representation will coincide with the finite field representation.
Moreover, we will also abuse notation by not distinguishing between
one-dimensional subspaces and elements in the field. This will imply that
$\alpha^i$ and $\beta \alpha^i$, $\beta \in \F_q \setminus \{ 0 \}$,
which are contained in the same one-dimensional subspace, will
be considered as the same element. Using this notation and omitting the zero vector
from the representation of a subspace, an $r$-dimensional subspace will
be defined by its $\frac{q^r-1}{q-1}$ nonzero elements.
Finally, if
$X = \{ \alpha^{i_1} , \ldots , \alpha^{i_{(q^r-1)/(q-1)}} \}$
is an $r$-dimensional subspace of $\F_q^n$ then we define the
$r$-dimensional subspace $\beta X$ by
$\beta X \deff \{ \beta \alpha^{i_1} , \ldots , \beta \alpha^{i_{(q^r-1)/(q-1)}} \}$,
for any $\beta \in \F_q^n \setminus \{ {\bf 0} \}$.

We define the following relation $E_r$ on the elements of $\cG_q(n,r)$.
\begin{equation}
\label{eq:equiv}
\text{For} ~~ X,~Y \in \cG_q(n,r), ~~ (X,Y) \in E_r ,~~ \text{if} ~~ Y = \alpha^j X~~ \text{for~some}~j~.
\end{equation}
The relation $E_r$ will be applied for $r=k$ and $r=k+1$,
where $n=2k+1$.

The following three lemmas are essential in the discussion which follows.
The first lemma can be readily verified.
\begin{lemma}
$E_r$ is an equivalence relation.
\end{lemma}
\begin{lemma}
\label{lem:full_cycle}
If $\gcd ( \frac{q^r-1}{q-1} , \frac{q^n-1}{q-1} ) =1$ then the number of
elements in an equivalence class of $E_r$ is~$\frac{q^n-1}{q-1}$.
\end{lemma}
\begin{proof}
Let $X$ be an $r$-dimensional subspace of $\cP_q(n)$ represented
as $X= \{ x_1 , x_2 , \ldots , x_{\frac{q^r-1}{q-1}} \}$. If $X = \alpha^j X$
for some $j$, $1 \leq j < \frac{q^n-1}{q-1}$ then clearly $X = \alpha^{j \ell} X$
for any positive integer $\ell$. Moreover, if such $j$ exists then
$x_i = \alpha^{j \frac{q^r-1}{q-1}} x_i$ for each $i$, $1 \leq i \leq \frac{q^r-1}{q-1}$.

Assume that there
exists an equivalence class of $E_r$ whose size is $j$, $j < \frac{q^n-1}{q-1}$.
This implies that $X = \alpha^j X$ and hence for each $i$, $1 \leq i \leq \frac{q^r-1}{q-1}$,
we have $x_i = \alpha^{j \frac{q^r-1}{q-1}} x_i$. Therefore, $\frac{q^n-1}{q-1}$
divides $j \frac{q^r-1}{q-1}$, but since $j < \frac{q^n-1}{q-1}$ it follows
that $\gcd ( \frac{q^r-1}{q-1} , \frac{q^n-1}{q-1} ) > 1$,
Hence, if $\gcd ( \frac{q^r-1}{q-1} , \frac{q^n-1}{q-1} ) =1$ then there
is no equivalence class of $E_r$ whose size is less than $\frac{q^n-1}{q-1}$.
\end{proof}
Finally, the following lemma is also trivial.
\begin{lemma}
$\gcd ( \frac{q^k-1}{q-1} , \frac{q^n-1}{q-1} ) =1$ if and only if
$\gcd (k,n) =1$.
\end{lemma}
\begin{cor}
\label{cor:gcd}
If $n=2k+1$ then
$\gcd ( \frac{q^k-1}{q-1} , \frac{q^n-1}{q-1} ) =1$
and $\gcd ( \frac{q^{k+1}-1}{q-1} , \frac{q^n-1}{q-1} ) =1$.
\end{cor}

We will modify the construction, based on necklaces,
for a cycle in the middle levels
of the Hamming graph to form a cycle in the middle levels
of the projective space graph $\cP_q(2k+1)$. Instead of necklaces we will
use the equivalence classes of the relation $E_r$. The construction
is implied by the following theorem.
\begin{theorem}
\label{thm:long_cycle}
Assume that there exists a sequence of subspace
of~$\F_q^{2k+1}$, $P=X_1 , Y_1 , X_2 , Y_2 , \ldots , Y_t , Y_t$,
satisfying the following requirements.
\begin{itemize}
\item[P.1] $X_i$, $1 \leq i \leq t$, is a $k$-dimensional
subspace of $\F_q^{2k+1}$.
\item[P.2] $Y_i$, $1 \leq i \leq t$, is a $(k+1)$-dimensional
subspace of $\F_q^{2k+1}$.
\item[P.3] All the $2t$ subspaces of $\F_q^{2k+1}$ in the sequence $P$ are contained in $2t$
different equivalence classes of $E_k$ and $E_{k+1}$.
\item[P.4] $X_i$ and $Y_i$, $1 \leq i \leq t$, are connected by an edge in $\cP_q(2k+1)$.
\item[P.5] $Y_i$ and $X_{i+1}$, $1 \leq i \leq t-1$, are connected by an edge in $\cP_q(2k+1)$.
\item[P.6] $\alpha^\ell X_1$ and
$Y_t$ are connected by an edge in $\cP_q(2k+1)$ and $\gcd ( \ell ,\frac{q^{2k+1}-1}{q-1})=1$.
\end{itemize}
Then there exists a cycle of length $2t\frac{q^{2k+1}-1}{q-1}$ in the middle levels
of $\cP_q(2k+1)$.
\end{theorem}
\begin{proof}
By P.1, P.2, P.3, P.4, and P.5, we have that $P$ is a path in the middle
levels of $\cP_q(2k+1)$.
By Lemma~\ref{lem:full_cycle} and Corollary~\ref{cor:gcd}, we have
that the number of subspaces in each equivalence class of $E_k$
and each equivalence class of $E_{k+1}$ is $\frac{q^{2k+1}-1}{q-1}$.
By Corollary~\ref{cor:gcd}, we also have that
the paths $P^{i \ell}$ and $P^{j \ell }$,
$0 \leq i < j \leq \frac{q^{2k+1}-1}{q-1}-1$,
where superscripts are taken modulo $\frac{q^{2k+1}-1}{q-1}$, are disjoint.
Taking P.6 into account too implies that $P^{i \ell} , P^{(i+1) \ell}$,
$0 \leq i \leq \frac{q^{2k+1}-1}{q-1}-1$,
is also a path in the middle levels of $\cP_q(2k+1)$.
Thus, $P^0 , P^\ell , P^{2 \ell} , \ldots , P^{\frac{q^{2k+1}-1}{q-1}-\ell}$
is a cycle in the middle levels of $\cP_q(2k+1)$.
\end{proof}
\begin{cor}
\label{cor:long_cycle}
If the requirements of Theorem~\ref{thm:long_cycle} are
satisfied with $t=\sbinomq{2k+1}{k} \frac{q-1}{q^{2k+1}-1}$
then there exists a Hamiltonian cycle in the middle levels
of $\cP_q(2k+1)$.
\end{cor}

The requirements of Theorem~\ref{thm:long_cycle} can be easily
satisfied with $t=\sbinomq{2k+1}{k} \frac{q-1}{q^{2k+1}-1}$,
for any $q$ and $k=1$. For this purpose we need the discussion
in~\cite{Sin38} in which the following theorem is proved.
\begin{theorem}
\label{thm:Singer}
Let $\alpha$ be a primitive element in $\F_q^3$.
Let $\alpha^{i_0} , \alpha^{i_1} , ~ \ldots ~, \alpha^{i_q}$,
$0 \leq i_0 < i_1 < ~ \cdots < i_q \leq q^2 +q$,
be $q+1$ elements in $\F_q^3$, for
which any two are linearly independent and any three
are linearly dependent. Then for any $\ell$ there exists
exactly one pair $\alpha^{i_r} ,~\alpha^{i_s}$, $r \neq s$,
such that $i_r - i_s \equiv \ell~ (\text{mod}~q^2+q+1)$.
\end{theorem}

Let $Y$ be any two-dimensional
subspace of $\F_q^3$. Clearly, $Y$ contains $q+1$ elements
and can be regarded as a line in the projective plane of order $q$.
Therefore, by Theorem~\ref{thm:Singer} there exists a $j$,
$0 \leq j < \frac{q^3-1}{q-1}$ such that
$\alpha^j$, $\alpha^{j+1} \in Y$. If $X$ is the
one-dimensional subspace which contains $\alpha^j$ then
we choose $P = X, Y$, and the requirements of Theorem~\ref{thm:long_cycle}
are satisfied with $t=\sbinomq{2k+1}{k} \frac{q-1}{q^{2k+1}-1}=1$ and
$\ell=1$. Therefore, by Corollary~\ref{cor:long_cycle}, we have constructed
a Hamiltonian cycle in the middle levels of $\cP_q(3)$. In fact
by Theorem~\ref{thm:Singer}, for each
$\ell$, $1 \leq \ell \leq q^2+q$ there exists a $j$,
$0 \leq j < \frac{q^3-1}{q-1}$ such that
$\alpha^j$,~$\alpha^{j+\ell} \in Y$, which implies the existence of
a few Hamiltonian cycles in the middle levels of $\cP_q(3)$.

\section{Two- and three-dimensional subspaces of $\F_q^5$}
\label{sec:2_3subspaces}

In this section we will explore some interesting properties
of two-dimensional subspaces and three-dimensional subspaces
of $\F_q^5$. These properties will be used in the next section
to construct many different Hamiltonian cycles in the middle levels of $\cP_q(5)$.
We start by considering the two equivalence relations $E_2$ on $\cG_q(5,2)$ and
$E_3$ on $\cG_q(5,3)$ as was defined earlier in~(\ref{eq:equiv}).
Our construction of Hamiltonian cycles in the middle levels
of $\cP_q(5)$ is based on properties which connect $E_2$ and $E_3$.
These properties are formulated in the following theorem whose proof
requires a sequence of lemmas which follow.

\begin{theorem}
\label{thm:two_three_E}
Each three-dimensional subspace of $\F_q^5$ contains exactly
$q+1$ elements from one equivalence class of the relation $E_2$ and exactly one element
from each other equivalence class of $E_2$.
\end{theorem}

The consequences of Theorem~\ref{thm:two_three_E} in terms
of the bipartite middle levels subgraph of $\cP_q(5)$
will be used in our construction.
Let $v$ be a vertex, representing a three-dimensional subspace,
in the middle levels of $\cP_q(5)$. A three-dimensional subspace
contains $\sbinomq{3}{2} =q^2+q+1$ two-dimensional
subspaces. Therefore, the degree of $v$ is $q^2+q+1$,
$q+1$ edges to $q+1$ distinct vertices representing one equivalence
class of $E_2$, and $q^2$ edges to $q^2$ vertices representing
the other $q^2$ equivalence classes of $E_2$. These are the basic
properties used in our construction of a Hamiltonian cycle
in the middle levels of $\cP_q(5)$.

\vspace{0.2cm}

We start with an immediate consequence of Lemma~\ref{lem:full_cycle}
related to $E_2$ and $E_3$.
\begin{cor}
\label{cor:size_E}
Each equivalence class of $E_2$ and each equivalence class
of $E_3$ contains $s \deff \frac{q^5-1}{q-1} =q^4 + q^3 +q^2 +q+1$ elements.
\end{cor}
\begin{cor}
\label{cor:num_E}
The number of equivalence classes of $E_2$ is $q^2+1$ and this
is also the number of equivalence classes of $E_3$.
\end{cor}
\begin{proof}
The number of two-dimensional subspaces of $\F_q^5$
is $\frac{(q^5-1)(q^4-1)}{(q^2-1)(q-1)}= s (q^2+1)$.
Thus, by Corollary~\ref{cor:size_E}, the number of equivalence classes
of $E_2$ is $q^2+1$. The same computation holds for~$E_3$.
\end{proof}

\begin{lemma}
\label{lem:3_ind}
$\alpha^0$, $\alpha^i$, and $\alpha^{2i}$,
$1 \leq i \leq s-1$,
are linearly independent elements in $\F_{q^5}$.
\end{lemma}
\begin{proof}
Assume the contrary, that $\alpha^0$, $\alpha^i$, and $\alpha^{2i}$,
are linearly dependent. This implies that $\alpha^0 , \alpha^i , \alpha^{2i} , \ldots , \alpha^{qi}$,
is a two-dimensional subspace and $\alpha^{(q+1)i}= \beta \alpha^0$ for
some $\beta \in \F_q$ and hence $\alpha^{(q+1)i} \in \F_q$.
Since $\alpha^0 , \alpha^s , \ldots , \alpha^{(q-2)s}$
are the elements of $\F_q$ in $\F_{q^5}$ it follows that
$\alpha^{(q+1)i}= \alpha^{\ell s}$, for some
$1 \leq \ell \leq q-1$. This implies that $q+1$ divides $\ell s$.
Since $s=\frac{q^5-1}{q-1} = q^4+q^3 + q^2 + q +1$
is follows that $\frac{\ell s}{q+1} = \ell (q^3 +q + \frac{1}{q+1})$
and hence $q+1$ divides $\ell$. But, since $1 \leq \ell \leq q-1$ this is impossible, a contradiction.
Thus, $\alpha^0$, $\alpha^i$, and $\alpha^{2i}$,
are linearly independent elements in $\F_{q^5}$.
\end{proof}

\begin{lemma}
\label{lem:num_pairs}
Each two-dimensional subspace $L$ has exactly $q^2+q$ distinct pairs
of the form $(x , y)$ such that
$L= \Span{ x , y  }$.
\end{lemma}
\begin{proof}
Let $L$ be a two-dimensional subspace.
$L$ is spanned by any two of its one-dimensional subspaces.
$L$ contains exactly $q+1$ distinct one-dimensional
subspaces. Hence, the
number of distinct ordered pairs of elements from $L$ is $(q+1)q$.
Thus, $L$ has exactly $q^2+q$ distinct pairs
of the form $(x , y)$ such that
$L= \Span{ x , y  }$.
\end{proof}

\begin{lemma}
\label{lem:distinct_diff}
If $(\alpha^t , \alpha^{t+i})$, $(\alpha^\ell , \alpha^{\ell+j})$,
$0 \leq t < \ell <s$, are two
distinct pairs of one-dimensional subspaces in a two-dimensional subspace $L$ then
$i \not\equiv j~(\text{mod}~s)$.
\end{lemma}
\begin{proof}
Assume the contrary, that $(\alpha^t , \alpha^{t+i})$, $(\alpha^\ell , \alpha^{\ell+j})$
are two distinct pairs of one-dimensional subspaces in a two-dimensional
subspace $L$ and $i \equiv j~(\text{mod}~s)$.
Since $L = \Span{\alpha^t , \alpha^{t+i}} = \Span {\alpha^\ell , \alpha^{\ell+i}}$
it follows that $L = \alpha^{\ell-t} L$. Since the number of one-dimensional
subspaces in a two-dimensional subspace is
$q+1$, it follows by Lemma~\ref{lem:full_cycle} that
$\gcd (q+1 , s) >1$, a contradiction.
\end{proof}

The set $\cD \deff \{ i ~:~ \Span{\alpha^0 , \alpha^i} \in \cG_2(5,2) \}$
contains $s-1=q^4+q^3+q^2+q$ elements. By Lemmas~\ref{lem:num_pairs}
and~\ref{lem:distinct_diff} it follows that each equivalence class
of $E_2$ can be represented by $q^2+q$ distinct elements from $\cD$
and the possible representatives of each equivalent class are
disjoint. Therefore, an equivalence class of $E_2$ will be denoted by
$[\Span{\alpha^0 , \alpha^i}]$, where each subspace of the equivalence class has
a distinct pair of elements of the form $(\alpha^t , \alpha^{t+i})$.

\begin{lemma}
\label{lem:triplet_exist}
If $\Span{ \alpha^t , \alpha^{t+i} }$ and $\Span{ \alpha^\ell , \alpha^{\ell+i} }$,
$1 \leq i \leq s-1$, are two distinct
two-dimensional subspaces in a three-dimensional subspace $Z$ of $\F_q^5$ then
$Z = \Span { \alpha^r , \alpha^{r+i} , \alpha^{r+2i} }$ for some $r$,
$0 \leq r \leq s-1$.
\end{lemma}
\begin{proof}
For any given $\beta \in \F_q \setminus \{ 0 \}$, we have
$\alpha^t + \beta \alpha^\ell = \alpha^m$ and
$\alpha^{t+i} + \beta \alpha^{\ell+i} = \alpha^{m+i}$,
for some~$m$, $0 \leq m \leq s-1$,
$m \not\in \{ t , \ell \}$. This implies that in $Z$ there
are two two-dimensional subspaces of the form
$$
\alpha^{b_0} , \alpha^{b_1} , \ldots , \alpha^{b_q}
$$
$$
\alpha^{b_0+i} , \alpha^{b_1+i} , \ldots , \alpha^{b_q+i}~.
$$
Any two distinct two-dimensional subspaces in a three-dimensional
subspace  intersect in exactly one one-dimensional subspace.
Therefore, w.l.o.g. we can assume that $b_1 = b_0 +i$.
Hence, $Z$ contains the points $\alpha^{b_0}$, $\alpha^{b_1}=\alpha^{b_0+i}$,
$\alpha^{b_1 + i}= \alpha^{b_0 +2i}$. By Lemma~\ref{lem:3_ind}, these
three points are linearly independent and thus
$Z = \Span { \alpha^b_0 , \alpha^{b_0+i} , \alpha^{b_0+2i} }$.
\end{proof}

\begin{lemma}
\label{lem:represent_3}
Each three-dimensional subspace of $\F_q^5$ can be represented
as $\Span { \alpha^r , \alpha^{r+i} , \alpha^{r+2i} }$ for some $r$ and $i$ such that
$0 \leq r \leq s-1$, $1 \leq i \leq s-1$.
\end{lemma}
\begin{proof}
A three-dimensional subspace $Z$ contains
$q^2+q+1$ distinct one-dimensional subspaces.
Therefore, there are $(q^2 +q+1)(q^2 +q) = q^4 + 2q^3 + 2q^2 +q$
ordered pairs of one-dimensional
subspaces of the form $(x , y)$, $x,y \in Z$, $x \neq y$.
$s=\frac{q^5-1}{q-1}=q^4+q^3+q^2+q+1$ and hence the set
$\{ i ~:~ (\alpha^t , \alpha^{t+i}) \in Z \times Z, ~ 0 <i < s \}$ contains at most
$q^4+q^3+q^2+q$ distinct integers. Hence, there exists at least two distinct
pairs of the form $( \alpha^t , \alpha^{t+i} )$, $( \alpha^\ell , \alpha^{\ell+i} )$,
$\alpha^t , \alpha^{t+i}, \alpha^\ell , \alpha^{\ell+i} \in Z$.
Hence, $Z$ contains two distinct two-dimensional subspaces of
the form $\Span{ \alpha^t , \alpha^{t+i} }$ and $\Span{ \alpha^\ell , \alpha^{\ell+i} }$.
Thus, by Lemma~\ref{lem:triplet_exist} the claim follows.
\end{proof}

\begin{lemma}
\label{lem:q+1}
If $Z$ is a three-dimensional subspace of $\F_q^5$ represented
as $\Span { \alpha^r , \alpha^{r+i} , \alpha^{r+2i} }$ for some~$r$ and $i$ such that
$0 \leq r \leq s-1$, $1 \leq i \leq s-1$, then it contains at least $q+1$ two-dimensional subspaces
from the equivalence class $[ \Span { \alpha^0 , \alpha^i } ]$ of $E_2$.
\end{lemma}
\begin{proof}
For any given $\beta \in \F_q \setminus \{ 0 \}$, we have
$\alpha^r + \beta \alpha^{r+i} = \alpha^{m_\beta}$ and
$\alpha^{r+i} + \beta \alpha^{r+2i} = \alpha^{m_\beta+i}$,
for some $m_\beta$, $0 \leq m_\beta \leq s-1$. Therefore,
$\Span { \alpha^r , \alpha^{r+i} , \alpha^{r+2i} }$ contains
the $q-1$ two-dimensional subspaces $\Span{ \alpha^{m_\beta},\alpha^{m_\beta+i}}$,
$\beta \in \F_q \setminus \{ 0 \}$, and the two two-dimensional subspaces
$\Span { \alpha^r , \alpha^{r+i} }$ and $\Span {\alpha^{r+i} , \alpha^{r+2i} }$.
These $q+1$ two-dimensional subspaces are
from the equivalence class $[ \Span { \alpha^0 , \alpha^i } ]$ of $E_2$.
\end{proof}

\begin{lemma}
\label{lem:Ex_q+1}
For a three-dimensional subspace $Z$ of $\F_q^5$
there exists exactly one $i$, $1 \leq i \leq s-1$,
such that $Z$
contains exactly $q+1$ two-dimensional subspaces
from the equivalence class $[ \Span { \alpha^0 , \alpha^i } ]$ of $E_2$.
\end{lemma}
\begin{proof}
By Lemma~\ref{lem:represent_3},
each three-dimensional subspace of $\F_q^5$ can be represented
as $\Span { \alpha^r , \alpha^{r+i} , \alpha^{r+2i} }$ for some $r$ and $i$ such that
$0 \leq r \leq s-1$, $1 \leq i \leq s-1$.
Hence, by Lemma~\ref{lem:q+1}, it contains
at least~$q+1$ two-dimensional subspaces
from the equivalence class $[ \Span { \alpha^0 , \alpha^i } ]$ of $E_2$.
This equivalence class of $E_2$ is the same for all the three-dimensional
subspaces which are contained in the same equivalence class of $E_3$.

Let $Z_1$ and $Z_2$ be two three-dimensional subspaces which
contain $q+1$ two-dimensional subspaces from the equivalence class
$[ \Span { \alpha^0 , \alpha^i } ]$ of $E_2$. Then by Lemma~\ref{lem:triplet_exist},
$Z_j = \Span { \alpha^{r_j} , \alpha^{{r_j}+i} , \alpha^{{r_j}+2i} }$ for some $r_j$,
$0 \leq r_j \leq s-1$, $j=1,2$. Hence, $Z_1$ and $Z_2$
are contained in the same equivalence class of $E_3$.
By Corollary~\ref{cor:num_E}, the number of equivalence classes of $E_3$
is $q^2+1$ and the number of equivalence classes of $E_2$ is $q^2+1$.
Therefore, by the above arguments and the pigeonhole principle,
if there exists a three-dimensional subspace of $\F_q^5$ which
contains $q+1$ two-dimensional
subspaces from two different equivalences classes of $E_2$ then there exists
a three-dimensional subspace of $\F_q^5$ which does not contain any
$q+1$ two-dimensional subspaces which are contained in one equivalence class of $E_2$.
This is a contradiction to Lemmas~\ref{lem:represent_3} and~\ref{lem:q+1}.
\end{proof}

In the sequel, an equivalence class of $E_3$ will be
represented by $[\Span { \alpha^0 , \alpha^i , \alpha^{2i} }]$, if each of its
three-dimensional subspaces have a subspaces of the form
$\Span { \alpha^r , \alpha^{r+i} , \alpha^{r+2i} }$.
By Lemma~\ref{lem:Ex_q+1} there is no ambiguity in using this
representation and this representation is unique.

\begin{lemma}
\label{lem:Ex_one}
A three-dimensional subspace in the equivalence class
$[\Span { \alpha^0 , \alpha^i , \alpha^{2i} }]$
contains at most one two-dimensional subspace from
each equivalence class of $E_2$ different from $[\Span { \alpha^0 , \alpha^i }]$.
\end{lemma}
\begin{proof}
Assume $Z$ is a three-dimensional subspace in the equivalence class
$[\Span { \alpha^0 , \alpha^i , \alpha^{2i} }]$.
If $Z$ has two two-dimensional subspaces of the form
$\Span{ \alpha^t , \alpha^{t+j} }$ and $\Span{ \alpha^\ell , \alpha^{\ell+j} }$,
$j \neq i$, then by Lemma~\ref{lem:triplet_exist},
$Z = \Span { \alpha^u , \alpha^{u+j} , \alpha^{u+2j} }$ for
some $u$, $0 \leq u \leq s-1$. This implies that
$Z \in [\Span { \alpha^0 , \alpha^i , \alpha^{2j} }]$
and hence by Lemma~\ref{lem:q+1}, it contains $q+1$ distinct
two-dimensional subspaces from $[\Span { \alpha^0 , \alpha^j }]$.
By Lemma~\ref{lem:Ex_q+1}, this implies that
$[\Span { \alpha^0 , \alpha^i }]=[\Span { \alpha^0 , \alpha^j }]$
and the lemma follows.
\end{proof}

\begin{proof}[{\bf Proof of Theorem~\ref{thm:two_three_E}}]
By Lemma~\ref{lem:Ex_q+1}, each three-dimensional subspace $Z$ of $\F_q^5$
contains $q+1$ distinct two-dimensional subspaces from exactly
one equivalence class of $E_2$. By Lemma~\ref{lem:Ex_one}, from each other
equivalence class of $E_2$ at most one two-dimensional
subspace is contained in~$Z$. Since a three-dimensional subspace
contains $q^2+q+1$ distinct two-dimensional subspaces it follows
that~$Z$ contains $q+1$ two-dimensional subspaces from one
equivalence class of $E_2$ and one two-dimensional subspace
from exactly $q^2$ different equivalence classes of $E_2$.
Since by Corollary~\ref{cor:num_E}, $E_2$ has exactly $q^2+1$ equivalence
classes the claim of the lemma follows immediately.
\end{proof}

%

Finally, we will use the following two lemmas in our construction.
\begin{lemma}
\label{lem:same_EQ1}
$[\Span{ \alpha^0 , \alpha^1 }] \neq [\Span{ \alpha^0 , \alpha^2 }]$.
\end{lemma}
\begin{proof}
Assume the contrary, that
$[\Span{ \alpha^0 , \alpha^1 }] = [\Span{ \alpha^0 , \alpha^2}]$.
THis implies that there exists a $j$, $5 \leq j \leq s-4$, such that
$\alpha^0 , \alpha^1 , \alpha^j , \alpha^{j+2} \in \Span{ \alpha^0 , \alpha^1 }$.
This implies that $\alpha^{j+2} = a + b \alpha$ and $\alpha^j = c + d \alpha$,
$a,b,c,d \in \F_q$.
Therefore, $\alpha^2 = \frac{a + b \alpha}{c + d \alpha}$ and
hence  $c \alpha^2 + d \alpha^3 =  a + b \alpha$. A nontrivial
solution for this equation implies that $\alpha$ is a root of a cubic
polynomial, contradicting the fact that lowest degree polynomial
for which $\alpha$ is a root has degree five. Thus,
$[\Span{ \alpha^0 , \alpha^1 }] \neq [\Span{ \alpha^0 , \alpha^2 }]$.
\end{proof}

\begin{lemma}
\label{lem:same_EQ2}
$[\Span{ \alpha^0 , \alpha^1 , \alpha^2}] \neq [\Span{ \alpha^0 , \alpha^1 , \alpha^3}]$.
\end{lemma}
\begin{proof}
Assume the contrary, that
$[\Span{ \alpha^0 , \alpha^1 , \alpha^2}] = [\Span{ \alpha^0 , \alpha^1 , \alpha^3}]$.
$\Span{ \alpha^0 , \alpha^1 }  = \{ \alpha^0 , \alpha^1 , \alpha^{i_1} , \alpha^{i_2} , \ldots , \alpha^{i_{q-1}} \}$
and hence $\Span{ \alpha^1 , \alpha^2 }  = \{ \alpha^1 , \alpha^2 , \alpha^{i_1+1} , \alpha^{i_2+1} , \ldots , \alpha^{i_{q-1}+1} \}$.
Since $\Span{ \alpha^0 , \alpha^1 , \alpha^2} \in [\Span{ \alpha^0 , \alpha^1 , \alpha^3}]$ it follows
that there exists a $j$, such that $\alpha^j , \alpha^{j+1} , \alpha^{j+3} \in \Span{ \alpha^0 , \alpha^1 , \alpha^2}$.
Clearly $j \in \{ 0,1, i_1 , i_2 , \ldots , i_{q-1} \}$. If $j=0$ or $j=1$ then $\alpha$ is a root
of a polynomial of degree less than five. Therefore, $j=i_r$ for some $1 \leq r \leq q-1$.
This implies that $\Span{ \alpha^0 , \alpha^2 }$ and $\Span{ \alpha^{i_r+1} , \alpha^{i_r+3} }$
are subspaces of $\Span{ \alpha^0 , \alpha^1 , \alpha^2}$.

Therefore, there are at least two distinct subspaces from $[\Span{ \alpha^0 , \alpha^2 }]$
which are contained in $\Span{ \alpha^0 , \alpha^1 , \alpha^2}$. Hence, by Theorem~\ref{thm:two_three_E}
we have that $[\Span{ \alpha^0 , \alpha^1 }]=[\Span{ \alpha^0 , \alpha^2 }]$ which contradicts
Lemma~\ref{lem:same_EQ1}.
\end{proof}

\section{Hamiltonian Cycles}
\label{sec:Hamilton}

In this section we will describe a construction for
Hamiltonian cycles in the middle levels of~$\cP_q(5)$
for any given power of a prime $q$. The construction will
be based on Theorem~\ref{thm:long_cycle} and
Corollary~\ref{cor:long_cycle}. To satisfy the requirements
of Theorem~\ref{thm:long_cycle} we will use the properties,
of two-dimensional and three-dimensional subspaces,
which were proved in the previous section. We will also show that the
number of such Hamiltonian cycles is very large. The construction
will be performed in a few steps. First, we will present a construction
which can satisfy the requirements of
Theorem~\ref{thm:long_cycle}. Unfortunately, we wouldn't be able to prove
this for all parameters.
Therefore, we will make a small modification in the construction
and present a construction which yields a Hamiltonian path.
Finally, we will make more modifications in the construction
which will enable us to construct many Hamiltonian cycles
in the middle levels of $\cP_q(5)$. In the sequel, given
a path (cycle) $P = X_1 , X_2 , \ldots , X_t$ which consists of
subspaces from $\cP_q(5)$ we define the path (cycle) $\beta P$, $\beta \in \F_{q^5}$,
by $\beta P \deff \beta X_1 , \beta X_2 , \ldots , \beta X_t$.

Let $\cA_1,\cA_2,\ldots,\cA_e$ be any order of the $e=q^2+1$ equivalence
classes of $E_3$, in which
$\cA_1 = [\Span{ \alpha^0 , \alpha^1 , \alpha^2}]$.
Let $\cB_1,\cB_2,\ldots,\cB_e$ be any order of the $q^2+1$ equivalence
classes of $E_2$, in which
$\cB_e =[ \Span{ \alpha^0 , \alpha^1 }]$. The following lemma
is an immediate consequence of  Theorem~\ref{thm:two_three_E}.

\begin{lemma}
\label{lem:simple_order}
There exists a path $P=U_1,V_1,U_2,V_2,\ldots,U_e,V_e$ in
the middle levels of $\cP_q(5)$, such that
$U_i \in \cA_i$, $V_i \in \cB_i$, $U_i \supset V_i$, for each
$1 \leq i \leq e$; $V_i \subset U_{i+1}$, for each
$1 \leq i \leq e-1$.
\end{lemma}

The path obtained via Lemma~\ref{lem:simple_order} can be chosen in
many different ways. We will restrict these choices in the following way.
Let $U_1 = \Span{ \alpha^0 , \alpha^1 , \alpha^2}$,
$V_1 = \Span{ \alpha^0 , \alpha^2}$, $U_e$ the
two-dimensional subspace $\Span{ \alpha^{\ell+1} , \alpha^{\ell+2} , \alpha^{\ell+4} }$
(clearly, by definition $V_{e-1} \subset U_e$), and $V_e$ the
unique two-dimensional subspace $\Span{ \alpha^{\ell+1} , \alpha^{\ell+2} }$
for which $V_e \subset U_e$. These choices also determine
the values of $\cA_1$, $\cB_1$, $\cA_e$, and $\cB_e$.
By Lemmas~\ref{lem:same_EQ1} and~\ref{lem:same_EQ2},
$U_1$, $V_1$, $U_e$, and $V_e$ are contained in different equivalence
classes of $E_2$ and $E_3$. Therefore, this order is possible and it is
well defined. The parameter~$\ell$ is
defined by this order, even so we should understand
that there are many possible choices to make this order and $\ell$ can
differ between these different possible choices. If $g = \gcd ( \ell , s)$,
and assuming $\ell \neq 0$ (we consider the simpler case $\ell=0$
at the end of this section), then let
$\Pi=P , \alpha^\ell P , \alpha^{2 \ell} P^0 , \ldots , \alpha^{(\frac{s}{g}-1)\ell} P$
be a sequence of subspaces (vertices) from the middle levels
of~$\cP_q(5)$.

\begin{lemma}
\label{lem:P_alP_path}
If $\ell \neq 0$ then
the sequence of vertices (subspaces) $P , \alpha^\ell P$
is a path in the middle levels of $\cP_q(5)$.
\end{lemma}
\begin{proof}
By definition, $P$ is a path in the middle
levels of $\cP_q(5)$ and hence also $\alpha^\ell P$
is a path in the middle levels of $\cP_q(5)$. The last vertex
in $P$ is $V_e=\Span{ \alpha^{\ell+1} , \alpha^{\ell+2} }$,
while the first vertex in~$\alpha^\ell P$ is
$\alpha^\ell U_1 =  \Span{ \alpha^\ell , \alpha^{\ell+1} , \alpha^{\ell+2}}$.
Clearly $V_e \subset \alpha^\ell U_1$ and hence $P , \alpha^\ell P$
is a path in the middle levels of $\cP_q(5)$.
\end{proof}
\begin{lemma}
\label{Pi_path}
If $\ell \neq 0$ then
the sequence of vertices (subspaces) $\Pi$ is a cycle
in the middle levels of $\cP_q(5)$ with $\frac{s}{g} 2e$ distinct vertices.
\end{lemma}
\begin{proof}
By Lemma~\ref{lem:P_alP_path}, we have that $\Pi$ is a cycle in
the middle levels of $\cP_q(5)$.
Since $g = \gcd ( \ell , s)$ it follows that the set
$\{ \alpha^{i \ell} U_j ~:~ 0 \leq i < \frac{s}{g} \}$ contains
$\frac{s}{g}$ distinct three-dimensional subspaces for each $j$,
$1 \leq j \leq e$.
Similarly, $\{ \alpha^{i \ell} V_j ~:~ 0 \leq i < \frac{s}{g} \}$ contains
$\frac{s}{g}$ distinct two-dimensional subspaces for each $j$,
$1 \leq j \leq e$. Therefore,
$\Pi$ is a cycle in the middle
levels of $\cP_q(5)$ with $\frac{s}{g} 2e$ distinct vertices.
\end{proof}
\begin{cor}
\label{cor:g=1}
If $\ell \neq 0$ and $g=1$ then $\Pi$ is a Hamiltonian cycle in
the middle levels of $\cP_q(5)$.
\end{cor}
\begin{cor}
If $\ell \neq 0$ then the
$g$ cycles $\Pi, ~ \alpha \Pi , ~ \alpha^2 \Pi , \ldots , \alpha^{g-1} \Pi$,
contain exactly all the vertices of the middle levels of $\cP_q(5)$.
Each vertex from the middle levels of $\cP_q(5)$ is contained exactly
once in one of these cycles.
\end{cor}

The condition $g=1$ in Corollary~\ref{cor:g=1} is satisfied in many
cases. For example, if $s$ is a prime integer then $g=1$. Moreover,
due to many choices in which the equivalence classes can be ordered,
and many choices for the relative shifts of the $U_i$'s and $V_j$'s,
such that the requirements of Theorem~\ref{thm:long_cycle}
are satisfied, it is reasonable
to believe that in many of them we will have $g=1$. But, we were not able to prove
this and hence we will slightly modify the construction to be able and obtain
a desired Hamiltonian cycle.

If $g>1$ then we consider $\Pi$ as a path in the middle
levels of $\cP_q(5)$ and define the sequence
of vertices (subspaces)
$\Pi , \alpha \Pi , \alpha^2 \Pi ,\ldots , \alpha^{g-1} \Pi$.

\begin{lemma}
\label{lem:Pi_alPi_path}
If $\ell \neq 0$ and $g \neq 1$ then
the sequence of vertices (subspaces) $\Pi , \alpha \Pi$ is a path
with $\frac{s}{g} 4 e$ distinct vertices
in the middle levels of $\cP_q(5)$.
\end{lemma}
\begin{proof}
By lemma~\ref{Pi_path}, we have that $\Pi$ is a path
in the middle levels of $\cP_q(5)$ and
therefore $\alpha \Pi$ is also a path
in the middle levels of $\cP_q(5)$. The last vertex (subspace) of $\Pi$
is $\alpha^{(\frac{s}{g}-1)\ell} V_e = \alpha^{(\frac{s}{g}-1)\ell} \Span{ \alpha^{\ell+1} , \alpha^{\ell+2} }
=\alpha^{\frac{s}{g} \ell} \Span{ \alpha^1 , \alpha^2 }= \Span{ \alpha^1 , \alpha^2 }$ and the first
vertex (subspace) of $\alpha \Pi$ is $\alpha U_1 = \Span{ \alpha^1 , \alpha^2 , \alpha^3}$.
Therefore, $\alpha^{(\frac{s}{g}-1)\ell} V_e \subset \alpha U_1$
and hence $\Pi , \alpha \Pi$ is a path
with $\frac{s}{g} 4 e$ distinct vertices
in the middle levels of $\cP_q(5)$.
\end{proof}
\begin{cor}
\label{cor:Ham_path}
If $\ell \neq 0$ and $g \neq 1$ then
$T \deff \Pi , \alpha \Pi , \alpha^2 \Pi ,\ldots , \alpha^{g-1} \Pi$
is a Hamiltonian path in the middle levels of $\cP_q(5)$.
\end{cor}

Corollary~\ref{cor:Ham_path} implies the existence of many Hamiltonian
paths in the middle levels of $\cP_q(5)$ since
by Theorem~\ref{thm:two_three_E}, we can order the equivalence
classes of $E_2$ and $E_3$ in almost any way that we want.
Anyway, our goal is still to show the existence and to construct Hamiltonian
cycles in the middle levels of $\cP_q(5)$. To achieve our goal we will modify
the Hamiltonian path which was obtained to form a Hamiltonian cycle.
This will be done by flipping the tail of the path several times,
preserving that the sequence of subspaces is still a path $\cP_q(5)$,
until a cycle is obtained.

\begin{lemma}
\label{lem:gN0_lN0H}
If $\ell \neq 0$ and $g \neq 1$ then there exists a Hamiltonian cycle in
the middle levels of $\cP_q(5)$.
\end{lemma}
\begin{proof}
Recall that $\Pi = \Span{ \alpha^0 , \alpha^1 , \alpha^2 },
\Span{ \alpha^0 , \alpha^2 } , \ldots , \Span{ \alpha^{\ell+1} ,
\alpha^{\ell+2} , \alpha^{\ell+4} }, \Span{ \alpha^{\ell+1} , \alpha^{\ell+2}},
\Span{ \alpha^\ell , \alpha^{\ell+1} , \alpha^{\ell+2}}$,
$\Span{ \alpha^\ell , \alpha^{\ell+2}}, \ldots,
\Span{ \alpha^1 , \alpha^2 , \alpha^4 }, \Span{ \alpha^1 , \alpha^2 }$.
For a path $Q$, let $r(Q)$ denote the \emph{reverse} of $Q$, i.e.,
$r(Q)$ is generated by taking the vertices of the path $Q$ in reverse order.
Let $\Upsilon_1$ be the prefix of the path $T$ which starts with the vertex
$\Span{ \alpha^0 , \alpha^1 , \alpha^2}$ and ends with the vertex
$\Span{ \alpha^1 , \alpha^2 , \alpha^4}$. If we denote
$T_1 \deff \Upsilon_1 , \Span{ \alpha^1 , \alpha^2 }, \alpha \Pi , \alpha^2 \Pi ,\ldots , \alpha^{g-1} \Pi$
then $T_2 \deff r(\Upsilon_1) , \Span{ \alpha^1 , \alpha^2 }, \alpha \Pi , \alpha^2 \Pi ,\ldots , \alpha^{g-1} \Pi$
is also a Hamiltonian path in the middle levels of~$\cP_q(5)$,
Note, that the first vertex (subspace) of~$r(\Upsilon_1)$ is
$\Span{ \alpha^1 , \alpha^2 , \alpha^4}$. Let $\Upsilon_2$ be the prefix
of $T_2$ which starts with the vertex $\Span{ \alpha^1 , \alpha^2 , \alpha^4}$
and ends with the vertex $\Span{ \alpha^2 , \alpha^3 , \alpha^4}$. If we denote
$T_2 = \Upsilon_2 , \Span{ \alpha^2 , \alpha^4} , \tilde{P}_2$
then $T_3= r(\Upsilon_2) , \Span{ \alpha^2 , \alpha^4} , \tilde{P}_2$.
This is the first step in modifying the Hamiltonian path $T_1$.
Before the $i$-th step we have the Hamiltonian path
$T_{2i-1} \deff \Upsilon_{2i-1} , \Span{ \alpha^{2i-1} , \alpha^{2i} },
\alpha^{2i-1} \Pi , \alpha^{2i} \Pi ,\ldots , \alpha^{g-1} \Pi$,
where $\Upsilon_{2i-1}$ starts with the vertex (subspace)
$\Span{ \alpha^{2i-2} , \alpha^{2i-1} , \alpha^{2i}}$ and ends with
$\Span{ \alpha^{2i-1} , \alpha^{2i} , \alpha^{2i+2}}$.
We form a new Hamiltonian path in the middle levels of $\cP_q(5)$,
$T_{2i} \deff r(\Upsilon_{2i-1}) , \Span{ \alpha^{2i-1} , \alpha^{2i} },
\alpha^{2i-1} \Pi , \alpha^{2i} \Pi ,\ldots , \alpha^{g-1} \Pi$.
Note, that the first vertex (subspace) of $r(\Upsilon_{2i-1})$ is
$\Span{ \alpha^{2i-1} , \alpha^{2i} , \alpha^{2i+2}}$.
Let $\Upsilon_{2i}$ be the prefix of $T_{2i}$ which starts with
$\Span{ \alpha^{2i-1} , \alpha^{2i} , \alpha^{2i+2}}$ and
ends with $\Span{ \alpha^{2i} , \alpha^{2i+1} , \alpha^{2i+2}}$
(which is the first vertex in $\alpha^{2i} \Pi$).
If we denote $T_{2i} = \Upsilon_{2i} ,  \Span{ \alpha^{2i} , \alpha^{2i+2}} , \tilde{P}_{2i}$
then $T_{2i+1} \deff r(\Upsilon_{2i}) ,  \Span{ \alpha^{2i} , \alpha^{2i+2}} , \tilde{P}_{2i}$.
The path $r(\Upsilon_{2i})$ starts with the vertex (subspace)
$\Span{ \alpha^{2i} , \alpha^{2i+1} , \alpha^{2i+2}}$
and ends with the vertex (subspace)
$\Span{ \alpha^{2i-1} , \alpha^{2i} , \alpha^{2i+2}}$. The
next vertex (subspace) in $T_{2i+1}$ is $\Span{ \alpha^{2i} , \alpha^{2i+2}}$
(which is the second vertex in $\alpha^{2i} \Pi$). The last vertex (subspace) in
$\alpha^{2i} \Pi$ is $\Span{ \alpha^{2i+1} , \alpha^{2i+2} }$. Therefore,
we can write
$T_{2i+1} = \Upsilon_{2i+1} , \Span{ \alpha^{2i+1} , \alpha^{2i+2} },
\alpha^{2i+1} \Pi , \alpha^{2i+2} \Pi ,\ldots , \alpha^{g-1} \Pi$
and we can continue with the recursive construction.
The $\frac{g-1}{2}$-th step is the last one.
The outcome of this step is the Hamiltonian path
$T_g = \Upsilon_g , \Span{ \alpha^{2i+1} , \alpha^{2i+2} },
\alpha^{2i+1} \Pi , \alpha^{2i+2} \Pi ,\ldots , \alpha^{g-1} \Pi$.
The first vertex in $T_g$ is
$\Span{ \alpha^g , \alpha^{g+1} , \alpha^{g+2}}$ and
the last vertex in $T_g$ is $\Span{ \alpha^g , \alpha^{g+1}}$.
Therefore $T_g$ is a Hamiltonian cycle.
\end{proof}

We assumed in the process that $\ell \neq 0$.
For $\ell=0$ we have a similar result.
\begin{lemma}
\label{lem:l=0}
If $\ell = 0$ then there exists a Hamiltonian cycle in
the middle levels of $\cP_q(5)$.
\end{lemma}
\begin{proof}
We modify the path $P$ defined in Lemma~\ref{lem:simple_order}
and the paragraph succeeding it as follows.
$U_1,V_1,U_2,V_2,\ldots,U_e , V_e$ is taken as in $P$;
note that $V_e = \Span{ \alpha^1 , \alpha^2 }$.
Let $\Pi=P , \alpha P , \alpha^2 P , \ldots , \alpha^{s-1} P$.
$\Pi$ is a Hamiltonian cycle in the middle levels
of $\cP_q(5)$ as proved in Lemma~\ref{lem:P_alP_path}.
\end{proof}

The fact that the order of the $q^2+1$ equivalence classes of $E_2$
and $q^2+1$ equivalence classes of $E_3$ is arbitrary (with some minor restrictions)
implies immediately more than $(q^2!)^2$ Hamiltonian cycles
obtained by our construction. This number can be increased
as the choices are even more flexible, but we
will omit the details as they are of less importance.
From our main
construction and its modifications as implied
by Corollary~\ref{cor:g=1}, Lemma~\ref{lem:gN0_lN0H},
and Lemma~\ref{lem:l=0} we infer the following theorem.
\begin{theorem}
\label{thm:many_H}
For each power of a prime $q$ there exists
a Hamiltonian cycles in the middle levels of $\cP_q(5)$.
\end{theorem}
\vspace{0.2cm}
Note that by Theorem~\ref{thm:many_H} there exists
a Hamiltonian cycle in the middle levels of $\cP_q(5)$,
but our discussion make it possible to produce many such
cycles. Unfortunately, we don't know how to count their number.

\section{Conclusion and problems for future research}
\label{sec:conclusion}

The well-known middle levels problem was considered for the
$k$-dimensional subspaces and the $(k+1)$-dimensional subspaces
of $\F_q^{2k+1}$. We have presented a technique based on cyclic
shifts of $k$-dimensional subspaces and $(k+1)$-dimensional
subspaces. The technique was applied successfully for
all $q$'s with $1 \leq k \leq 2$.
Applying the technique for larger values of $k$ to find Hamiltonian cycles
or long cycles in the middle levels of $\cP_q(2k+1)$ is the main goal
for future research. We note that if $k > 2$ then the degree of a vertex $v$,
representing a $(k+1)$-dimensional subspace of $\cP_q(2k+1)$,
is $\frac{q^{k+1}-1}{q-1}$. The number of equivalence classes
of $E_k$ will be $\sbinomq{2k+1}{k} \frac{q-1}{q^{2k+1}-1}$ which
is much larger than $\frac{q^{k+1}-1}{q-1}$ and hence $v$ won't be connected
with an edge to a vertex from each equivalence class of $E_k$.
Thus, we won't have a result similar to Theorem~\ref{thm:two_three_E}
and hence constructing the Hamiltonian path will be more difficult,
even though it looks that there is enough freedom and we have the following conjecture.

\begin{conj}
For any integer $k \geq 1$, and for any
power of a prime $q$, there exists a Hamiltonian cycle
in the middle levels of $\cP_q(2k+1)$.
\end{conj}


\begin{center}
{\bf Acknowlegment}
\end{center}
The author would like to thank Ronny Roth for many valuable
discussions which have contributed to the content of the paper.
He also thanks two anonymous reviewers whose careful reading
and well spotted comments contribute to amend this paper.


\end{document}